\documentclass[12pt]{article}

\usepackage{amssymb,amsmath,amsfonts}
\usepackage{enumerate}
\usepackage{graphicx,color,enumitem}
\usepackage{calrsfs}	
\usepackage{amsthm}	
\usepackage{cite}

\newtheorem{theorem}{Theorem}[section]
\newtheorem{lemma}[theorem]{Lemma}
\newtheorem{proposition}[theorem]{Proposition}
\newtheorem{corollary}[theorem]{Corollary}

\theoremstyle{definition}
\newtheorem{definition}[theorem]{Definition}
\newtheorem{example}[theorem]{Example}

\newtheorem*{acknowledgements*}{Acknowledgements}

\theoremstyle{remark}
\newtheorem{remark}[theorem]{Remark}

\numberwithin{equation}{section}

\newcommand{\nc}{\newcommand}
\nc{\R}{{\bf R}}
\nc{\cl}{\mbox{\rm cl}\,} 
\nc{\cls}{ \mbox{{\scriptsize {\rm cl}}}\,} 
\nc{\conv}{\mbox{\rm conv}\,} 
\nc{\cone}{\mbox{\rm coco}\,} 
\nc{\lin}{\mbox{\rm lin}\,} 
\nc{\rb}{\mbox{\rm rb}\,}
\nc{\ri}{\mbox{\rm ri}\,}
\nc{\inter}{\mbox{\rm int}\,}
\nc{\bd}{\mbox{\rm bd}\,}
\nc{\relbd}{\mbox{\rm rlbd}\,}
\nc{\epi}{\mbox{\rm epi}\,}
\nc{\gph}{\mbox{\rm gph}\,}
\nc{\rge}{\mbox{\rm rge}\,}
\nc{\rgel}{\mbox{\rm {\scriptsize rge}}\,}
\nc{\sepi}{\mbox{\rm {\scriptsize epi}}\,}
\nc{\sdom}{\mbox{\rm {\scriptsize dom}}\,}
\nc{\sgph}{\mbox{\rm {\scriptsize gph}}\,}
\nc{\dom}{\mbox{\rm dom}\,}
\nc{\para}{\mbox{\rm par}\,}
\nc{\crit}{\mbox{\rm crit}\,}
\nc{\dist}{\mbox{\rm dist}\,}
\nc{\kernal}{\mbox{\rm ker}\,}
\nc{\supp}{\mbox{\rm supp}\,}
\nc{\lip}{\mbox{\rm lip}\,}
\nc{\aff}{\mbox{\rm aff}\,}
\nc{\sbd}{\mbox{\rm {\scriptsize  bd}}\,}

\nc{\id}{\mbox{\rm Id}}

\providecommand{\tr}{\mathop{\rm tr}\nolimits}
\providecommand{\rk}{\mathop{\rm rk}\nolimits}

\providecommand{\Diag}{\mathop{\rm Diag}\nolimits}

\begin{document}

\title{Approximating functions on stratified sets\footnote{The research of D. Drusvyatskiy was made with Government support under and awarded by DoD, Air Force Office of Scientific Research, National Defense Science and Engineering Graduate (NDSEG) Fellowship, 32 CFR 168a. M.~Larsson gratefully acknowledges funding from the European Research Council under the European Union's Seventh Framework Programme (FP/2007-2013) / ERC Grant Agreement n. 307465-POLYTE.}}
\author{D. Drusvyatskiy\thanks{Department of Mathematics, University of Washington, Seattle, WA 98195-4350 and Department of Combinatorics and Optimization, University of Waterloo, Waterloo, Ontario, Canada N2L 3G1, email: ddrusv@uw.edu} \quad\quad  M. Larsson\thanks{ETH Zurich, Department of Mathematics, R\"amistrasse 101, CH-8092, Zurich, Switzerland, email: martin.larsson@math.ethz.ch}}
\date{}



\maketitle

\begin{abstract}
We investigate smooth approximations of functions, with prescribed gradient behavior on a distinguished stratified subset of the domain. As an application, we outline how our results yield important consequences for a recently introduced class of stochastic processes, called the matrix-valued Bessel processes.
\\[2ex]
\end{abstract}

\section{Introduction} 
Nonsmoothness arises naturally in many problems of mathematical analysis. A conceptually simple way to alleviate the inherent difficulties involved is by smoothing. A classical result in this direction shows that any continuous function $f$ on $\R^n$ can be uniformly approximated by a $C^{\infty}$-smooth function $g$. See for example \cite[Theorem 10.16]{Lee:2003}. In light of this, it is natural to ask the following question. Can we, in addition, guarantee that such an approximating function $g$ satisfies ``useful'' properties on a distinguished nonsmooth subset $Q$ of $\R^n$?
In the current work, we consider sets $Q$ that  are {\em stratified} into finitely many smooth manifolds $\{M_i\}$. The ``useful'' property we would like to ensure is that the gradient of the approximating function $g$ at any point in $Q$ is tangent to the manifold containing that point, that is the inclusion 
\[\nabla g(x)\in T_x M_i \textrm{ holds for any index }  i \textrm{ and any point } x \in M_i.\]
This requirement can be thought of as a kind of Neumann boundary condition. In the language of stratification theory, we would like to approximate $f$ by a smooth function $g$ so that $\nabla g$ is a {\em stratified vector field}; see for example \cite{mather}. To the best of our knowledge, such a question has never been explicitly considered, and therefore we strive to make our development self-contained.

We provide an intuitive and transparent argument showing the existence of a $C^1$-smooth approximating function satisfying the tangency condition, provided that the partitioning manifolds yield a {\em Whitney (a)-regular $C^2$-stratification}.  In particular, our techniques are applicable for all semi-algebraic sets---those sets that can be written as a union of finitely many sets, each defined by finitely many polynomial inequalities. For more details on semi-algebraic geometry, see for example \cite{ARAG,Coste-semi}. Guaranteeing a higher order of smoothness for the approximating function, even when the partitioning manifolds are of class $C^{\infty}$, seems fairly difficult, with the curvature of the manifolds entering the picture. Nevertheless, we identify a simple and easily checkable, though stringent, condition on the stratification ---{\em normal flatness} (Definition~\ref{A:1})---that bypasses such technical difficulties and allows us to guarantee that whenever the partitioning manifolds are $C^{\infty}$-smooth, so is the approximating function. 

At first sight, the {\em normal flatness} condition is deeply tied to polyhedrality. However, we prove that this condition satisfies the so-called {\em Transfer Principle} investigated for example in \cite{prox,man,man2,spec_id,der,lag,send,high_order}.
Consequently this condition holds for a number of important subsets of matrix spaces, and our strongest results become applicable. This allows us to apply our techniques to the study of a class of stochastic processes called the \emph{matrix-valued Bessel processes}, introduced in~\cite{Larsson:2012}. We give an informal outline of how our results constitute a key component needed to obtain a good description of the law of the process, and how they enable powerful uniqueness results to become available. Indeed, this was the original motivation for the current work.
The main approximability results of the current paper are stated in terms of the uniform metric. In Section~\ref{app:sob}, we extend these results to weighted Sobolev norms, which is perhaps more natural given the boundary-value feel of the problem at hand.

The outline of the manuscript is as follows. In Section~\ref{S:M}, we record basic notation that we will use throughout, and state the main results of the paper. Section~\ref{S:it} contains the proofs of the main results. In Section~\ref{S:A1}, we discuss the {\em Transfer Principle} and how it relates to stratifications. In Section~\ref{app:sob} we prove that, under reasonable conditions, smooth functions satisfying the tangency condition are dense in appropriate weighted Sobolev spaces. Finally in Section~\ref{S:M+}, we outline an application of our results to matrix-valued Bessel processes.

\section{Basic notation and summary of main results} \label{S:M}
Throughout, the symbol $|\cdot|$ will denote the standard Euclidean norm on $\R^n$ and the absolute value of a real number, while $\|\cdot\|$ will denote the induced operator norm on the space of linear operators on $\R^n$. A function $f\colon Q\to\R$, defined on a set $Q\subset\R^n$, is called $\kappa$-{\em Lipschitz continuous} (for some real $\kappa \geq 0$) if the inequality
\[
|f(x)-f(y)|\leq\kappa|x-y| \quad\textrm{ holds for all } x,y\in Q.
\]
The infimum of $\kappa\geq 0$ satisfying the inequality above is the {\em Lipschitz modulus} of $f$, and we denote it by $\mbox{\rm lip}\,f$. Thus we have $$\lip f=\sup_{x,y\in Q}\frac{|f(x)-f(y)|}{|x-y|}.$$ Whenever $\lip f$ is finite, we will say that $f$ is Lipschitz continuous. For notational convenience, $1$-Lipschitz continuous functions will be called {\em non-expansive}. 
A function $f\colon Q\to\R$ is said to be {\em locally Lipschitz continuous} if around each point $x\in Q$, there exists a neighborhood $U$ so that the restriction of $f$ to $Q\cap U$ is Lipschitz continuous.

Given any set $Q\subset\R^n$ and a mapping $f\colon Q\to \widetilde{Q}$, where $\widetilde{Q}\subset\R^m$, we say that $f$ is $C^p$-{\em smooth} if for each point $\bar{x}\in Q$, there is a neighborhood $U$ of $\bar{x}$ and a $C^p$-smooth mapping $\widehat{f}\colon \R^n\to\R^m$ that agrees with $f$ on $Q\cap U$. Throughout the manuscript, it will always be understood that $p$ lies in $\{1,2,\ldots,\infty\}$.

The following definition is standard.

\begin{definition}[Smooth manifold]
Consider a set $M\subset\R^n$. We say that $M$ is a $C^p$ {\em manifold} (or ``embedded submanifold'') {\em of dimension $r$} if for each point $\bar{x}\in M$, there is an open neighborhood $U$ around ${\bar{x}}$ such that $M\cap U=F^{-1}(0)$ for some $C^p$-smooth map $F\colon U\to\R^{n-r}$, with the derivative ${\mathrm D}F(\bar{x})$ having full rank. 
\end{definition}

The tangent space of a manifold $M$ at a point $x\in M$ will be denoted by $T_xM$, while the normal space will be denoted by $N_xM$. We will consider $T_xM$ and $N_xM$ as embedded subspaces of $\R^n$.

In the current work, we will be interested in subsets of $\R^n$ that can be decomposed into finitely many smooth manifolds satisfying certain compatibility conditions. Standard references on stratification theory are \cite{gor_mac,pflaum}.

\begin{definition}[Stratifications]\label{defn:whit}
{\rm
A $C^p$-{\em stratification} ${\mathcal A}$ of a set $Q\subset\R^n$ is a partition of $Q$ into finitely many nonempty $C^p$ manifolds (not necessarily connected), called {\em strata}, satisfying the following compatibility condition. 
\begin{itemize}[leftmargin=0.5cm]
\item[] {\bf Frontier condition:} For any two strata $L$ and $M$, the implication
\[
L\cap \cl M \neq \emptyset\quad\Longrightarrow\quad L\subset (\cl M)\setminus M \quad \textrm{ holds}.
\]
\end{itemize}
A $C^p$-stratification ${\mathcal A}$ is said to be {\em Whitney (a)-regular}, provided that the following condition holds. 
\begin{itemize}[leftmargin=0.5cm]
\item[] {\bf Whitney condition (a):} For any sequence of points $x_k$ in a stratum $M$ converging to a point $\bar{x}$ in a stratum $L$, if the corresponding normal vectors $v_k \in N_{x_k}M$ converge to a vector $v$, then the inclusion $v \in N_{\bar{x}}L$ holds.
\end{itemize} }
\end{definition}

\begin{remark}
{\rm
We should emphasize that though we require stratifications to be comprised of finitely many manifolds (strata), each such manifold may have infinitely many connected components. This being said, it is worth noting that sometimes in the literature the term stratum (unlike our convention) refers to each connected component of the partitioning manifolds.}
\end{remark}

It is reassuring to know that many important sets fall in this class: every subanalytic set admits a Whitney (a)-regular $C^p$-stratification, for any finite $p$, as does any definable set in an arbitrary o-minimal structure. In particular, this is true for semi-algebraic sets. For a discussion, see for example \cite{DM}.

Given a stratification $\mathcal A$, the frontier condition induces a (strict) partial order~$\prec$ on~$\mathcal A$, defined by
\[
L\prec M \quad\Longleftrightarrow\quad L\subset (\cl M)\setminus M.
\]
A stratum $M\in\mathcal A$ is {\em minimal} if there is no stratum $L\in\mathcal A$ with $L\prec M$. The {\em depth} of $\mathcal A$ is the maximal integer $m$ such that there exist strata $M_1,\ldots,M_m$ with $M_1\prec M_2\prec\cdots\prec M_m$.

We are now ready to state the main results of this paper. Their proofs are given in Section~\ref{S:it}.

\begin{theorem}[Approximation on stratified sets]\label{thm:main}
 Consider a closed set $Q\subset\R^n$ along with a $C^2$-stratification $\mathcal A$ of~$Q$, and let $f:\R^n\to\R$ and $\varepsilon:\R^n\to(0,\infty)$ be continuous functions. Then there exists a function $g\colon\R^n\to\R$ satisfying the following properties.
\begin{itemize}[leftmargin=0.5cm]
\item[] {\bf Uniform closeness:} The inequality $|f(x)-g(x)|< \varepsilon(x)$ holds for all $x\in \R^n$.
\item[] {\bf Differentiability and Lipschitzness:} The function $g$ is locally Lipschitz continuous and differentiable. Moreover, if $f$ is Lipschitz continuous, then so is $g$ and the estimate $\lip g \le 11^{m+1} ~\lip f$ holds, where $m$ is the depth of~$\mathcal A$.
\item[] {\bf Tangency condition:} For any stratum $M\in\mathcal{A}$ and any point $x\in M$, the inclusion 
\[\nabla g(x)\in T_xM\quad \textrm{ holds}.\]
\item[] {\bf Support:} Given a neighborhood $V_1$ of the support of $f$, we may choose $g$ so that its support is contained in $V_1$.
\item[] {\bf Agreement on full-dimensional strata:} If $f$ is $C^1$-smooth, then given a neighborhood $V_2$ of the set $$\bigcup_{M\in\mathcal{A},~\dim M<n} M,$$ we may choose $g$ so that it coincides with $f$ outside of $V_2$.
\end{itemize}
If $\mathcal{A}$ is a Whitney (a)-regular $C^2$-stratification, then we may in addition to the above properties ensure that $g$ is $C^1$-smooth.
\end{theorem}

\begin{remark}
{\rm
Note that in the above theorem, $f$ and $g$ are defined on all of $\R^n$, as opposed to only on~$Q$. In particular, any Lipschitz property should be understood with this in mind. Furthermore, the bound $\lip g\leq 11^{m+1}\ \lip f$ will be important for us because the multiplier $11^{m+1}$ {\em only} depends on the depth of~$\mathcal A$.

It is also worth mentioning that the hypothesis on $\varepsilon$ can be weakened to lower semicontinuity. Indeed, it is easy to show that any strictly positive, lower semicontinuous function on $\R^n$ can be bounded below by a strictly positive continuous function.}
\end{remark}

\begin{remark}
Theorem~\ref{thm:main} immediately implies that any Whitney $(a)$-regular $C^2$-stratification $\mathcal{A}$ of a set $Q\subset\R^n$, where $\mathcal{A}$ contains at least one nonzero dimensional stratum, admits a vector field that is not identically zero, continuous, conservative, and stratified. This nicely complements the development in \cite[Chapter~1]{MH}.
\end{remark}

Observe that even under the Whitney condition (a), the approximating function $g$, guaranteed to exist by Theorem~\ref{thm:main},  is only $C^{1}$-smooth. To guarantee a higher order of smoothness (under minimal conditions), it seems that one needs to impose stronger requirements, both on the stratification and on the curvature of the strata. Below we identify a simple, though stringent, condition which bypasses such technical difficulties. We begin with some notation.

Given a set $Q\subset\R^n$ and a point $x\in \R^n$, the {\em distance} of $x$ to $Q$ is
\[
d(x,Q):=\inf_{y\in Q}|x-y|,
\]
and the {\em projection} of $x$ onto $Q$ is
\[
P_Q(x):=\{y\in Q:|x-y|=d(x,Q)\}.
\]
Note that $P_Q(x)$ may be empty, a singleton, or it may contain multiple points.  When $P_Q(x)$ is a singleton, we will abuse notation slightly and write $P_Q(x)$ for the point it contains. A crucial fact for us will be that any $C^{p+1}$ manifold $M$ (for $p=1,\ldots,\infty$) admits a neighborhood on which the projection mapping $P_M$ is single-valued and $C^p$-smooth; see for example \cite[Lemma~4]{alt_proj} or \cite{distsq}. Indeed, this is the reason why throughout the manuscript we will be concerned with stratifications by manifolds that are at least of class $C^2$. See Subsection~\ref{S:tn} for more details.

\begin{definition}[Normally flat stratification]\label{A:1}
A $C^{p+1}$-{\em stratification} of a subset $Q$ of $\R^n$ is said to be {\em normally flat} if 
for any two strata $L,M \in\mathcal A$ with $L\prec M$, there are neighborhoods $V$ of $L$ and $U$ of $M$ so that the equality
\[
P_L(x) = P_L \circ P_M(x) \quad \textrm{holds for all}\quad x\in V\cap U.
\]
\end{definition}

The normal flatness condition is of course quite strong. In particular, normally flat stratifications automatically satisfy Whitney's condition~(a); see Proposition~\ref{P:A1Wh}. Nonetheless, this condition does hold in a number of important situations. Section~\ref{S:A1} contains a detailed analysis, in particular showing that polyhedral sets and {\em spectral lifts} of {\em symmetric polyhedra}, admit normally flat $C^{\infty}$-stratifications (Proposition~\ref{P:polflat} and Theorem~\ref{P:A1sp}).

The next result is a strengthened version of Theorem~\ref{thm:main}, under the normal flatness condition.

\begin{theorem}[Approximation on normally flat stratifications] \label{T:smooth}
Consider a closed set $Q\subset\R^n$ along with a normally flat $C^{p+1}$-stratification $\mathcal A$ of~$Q$, and let $f:\R^n\to\R$ and $\varepsilon:\R^n\to(0,\infty)$ be continuous functions. Then there exists a function $g:\R^n\to\R$ satisfying the following properties.
\begin{itemize}[leftmargin=0.5cm]
\item[] {\bf Uniform closeness:} The inequality $|f(x)-g(x)|< \varepsilon(x)$ holds for all $x\in \R^n$.
\item[] {\bf Smoothness and Lipschitzness:} The function $g$ is $C^p$-smooth. Moreover, if $f$ is Lipschitz continuous, then so is $g$ and the estimate $\lip g \le 11^{m+1} ~\lip f$ holds, where $m$ is the depth of~$\mathcal A$.
\item[] {\bf Enhanced tangency condition:} Each stratum $M\in\mathcal{A}$ has a neighborhood $U$ so that equality
$$g(x)=g\circ P_M(x)\quad \textrm{ holds for all } x\in U.$$
\item[] {\bf Support:} Given a neighborhood $V_1$ of the support of $f$, we may choose $g$ so that its support is contained in $V_1$.
\item[] {\bf Agreement on full-dimensional strata:} If $f$ is $C^p$-smooth, then given a neighborhood $V_2$ of the set $$\bigcup_{M\in\mathcal{A},~\dim M<n} M,$$ we may choose $g$ so that it coincides with $f$ outside of $V_2$.
\end{itemize}
\end{theorem}

In particular, the enhanced tangency condition of Theorem~\ref{T:smooth} directly implies that for any stratum $M\in\mathcal{A}$, any point $x\in M$, and any normal direction $v\in N_xM$, the function $\R\ni t\mapsto g(x+tv)$ is constant in a neighborhood of the origin. Consequently, we have
$${\mathrm D}^{(k)}g(x)[v,\ldots,v]=0, \quad \textrm{ for any } k=1,\ldots,p,$$
where ${\mathrm D}^{(k)}g(x)[v,\ldots,v]$ is the $k$'th order directional derivative of $g$ at $x$ in direction~$v$. This is a significant strengthening of the tangency condition in Theorem~\ref{thm:main}.

\section{The iterative construction and proofs of main results} \label{S:it}

Our method for proving results such as Theorems~\ref{thm:main} and~\ref{T:smooth} relies on an inductive procedure, where the original function $f$ (possibly after some initial smoothing) is first modified on the minimal strata, then on the strata whose frontiers consist of minimal strata, and so on. At each step, the resulting function approximates $f$ and has the desired properties on all strata preceding (in the sense of the partial order~$\prec$) the stratum currently under consideration. This allows the induction to continue. When no more strata remain, the desired properties hold on the entire stratified set. In this section we give a detailed description of this iterative construction, in particular the induction step. The reason for bringing out some of the details, as opposed to hiding them inside proofs, is that the same general approach can sometimes be used in specific situations to obtain further properties of the approximating function. An example of this will be discussed in Section~\ref{S:M+}.

The induction step has two crucial ingredients, namely (1)~finding a suitable \emph{tubular neighborhood} of the stratum where the current function is to be modified, and (2)~an {\em interpolation function} constructed using this tubular neighborhood. These two objects are described in Subsections~\ref{S:tn} and~\ref{S:if}. Then, in Subsection~\ref{S:it1}, the induction step is described in detail, laying the groundwork for the proofs of the main results, given in Subsection~\ref{S:pf}.

\subsection{The tubular neighborhood} \label{S:tn}
In this subsection, we follow the notation of \cite[Section 10]{Lee:2003}. We should stress that this subsection does not really contain any new results. Its only purpose is to record a number of observations needed in the latter parts of the manuscript.

For a $C^{p+1}$ manifold $M\subset\R^n$, the {\em normal bundle} of $M$, denoted by $NM$, is the set
\[
NM:=\{(x,v)\in \R^n\times\R^n:x\in M, v\in N_xM\}.
\]
It is well-known that the normal bundle $NM$ is itself a $C^p$ manifold. 
Consider the mapping $E:NM\to \R^n$ defined by
\[
E(y,v) = y+v.
\]
A {\em tubular neighborhood} $U$ of $M$ is by definition a $C^p$ diffeomorphic image under $E$ of an open subset $V\subset NM$ having the form 
\begin{equation} \label{eq:tube}
V = \{(y,v) \in NM : | v| < \delta(y)\},
\end{equation}
for some continuous function $\delta\colon M\to (0,\infty)$.
It is standard that tubular neighborhoods of $M$ always exist; see for example \cite[Theorem~10.19]{Lee:2003}. Furthermore, there always exists a tubular neighborhood on which the projection mapping $P_M$ is single-valued and $C^p$-smooth; see \cite[Lemma~4]{alt_proj}. Then for any point $\bar{x}$ lying in $M$, the derivative of the projection $\mathrm D P_M(\bar{x})$ is simply the orthogonal projection onto the tangent space $T_{\bar{x}}M$. It is worth noting that whenever the manifold $M$ is only $C^1$ smooth, the projector $P_M$ may easily fail to be single-valued on any neighborhood of $M$; this is the main reason why throughout the manuscript we consider $C^{p+1}$ manifolds for $p=1,\ldots,\infty$.

The following observation shows that we may squeeze the tubular neighborhood inside any given open set containing the manifold.

\begin{lemma}[Existence of a tubular neighborhood]\label{L:exis_tube}
Consider a $C^{p+1}$ manifold $M\subset\R^n$, as well as an arbitrary neighborhood $V$ of $M$. Then there exists a tubular neighborhood
\[
U_\delta=\left\{y+v\in\R^n: y\in M,\ v\in N_yM,\ |v|<\delta(y)\right\}
\]
satisfying $U_\delta\subset V$. Furthermore for any real $\kappa >0$, we may ensure that $\delta$ is $\kappa$-Lipschitz continuous.
\end{lemma}

\begin{proof}
By intersecting $V$ with some tubular neighborhood of $M$, we may assume that $V$ is the diffeomorphic image of some neighborhood of the zero-section $M\times \{0\}$ in $NM$. Now, for any point $y\in M$, define
\[
\delta(y) := \sup\{ 0\le\epsilon\le 1 : V_\epsilon(y) \subset V \},
\]
where
\[
V_\epsilon(y) := \{z+v \in \R^n : z\in M,\ v\in N_zM,\ |y-z|<\epsilon,\ |v|<\epsilon\}.
\]
Clearly $\delta$ is strictly positive and the inclusion $U_{\delta}\subset V$ holds.  We now show that $\delta$ is non-expansive. To see this, first note that whenever the inequality $|x-y|\ge \delta(x)$ holds, we have $\delta(x)-\delta(y)\le |x-y|$ trivially. On the other hand, if we have $|x-y|\le\delta(x)$, an application of the triangle inequality shows that the inclusion $V_\epsilon(y)\subset V_{\delta(x)}(x)$ is valid for $\epsilon = \delta(x)-|x-y|$. We deduce $\delta(y)\ge \epsilon$, and hence the inequality $\delta(x)-\delta(y)\le |x-y|$ is also valid in this case. Interchanging the roles of $x$ and $y$ gives the non-expansive property. Finally, replacing $\delta$ with $\kappa\delta$, if need be, ensures that $\delta$ is $\kappa$-Lipschitz continuous.
\end{proof}

In fact in Lemma~\ref{L:exis_tube}, we may ensure that $\delta$ is $C^{\infty}$-smooth. To see this, we need the following result, which has classical roots.

\begin{lemma}[Approximation of Lipschitz functions]\label{L:approx_lip}
Consider any $C^{p+1}$ manifold $M\subset\R^n$ and a function $f\colon M\to\R$ that is Lipschitz continuous. Then for any continuous function $\varepsilon\colon M\to (0,\infty)$ and a real $r >0$, there exists a Lipschitz continuous, $C^\infty$-smooth function $\widehat{f}\colon M\to\R$ satisfying 
\[|f(x)-\widehat{f}(x)|<\varepsilon(x) \textrm{ for all } x\in M,\]
with $\lip \widehat{f}\leq\lip f +r$.
\end{lemma}

\begin{proof}
The proof proceeds by extending the functions $f$ and $\varepsilon$ to an open neighborhood of $M$, using standard approximation techniques on this neighborhood, and then restricting the approximating function back to $M$. To this end, for any real $\epsilon >0$ there exists an open neighborhood $U$ of $M$ so that the projection $P_M$ is $(1+\epsilon)$-Lipschitz continuous on $U$. Consider now the functions $\widetilde{f}\colon U\to\R$ and $\widetilde{\varepsilon}\colon U\to (0,\infty)$, defined by $\widetilde{f}(x)=f(P_M(x))$ and $\widetilde{\varepsilon}(x)=\varepsilon(P_M(x))$. Observe that $\widetilde{f}$ agrees with $f$ on $M$ and $\widetilde{\varepsilon}$ agrees with $\varepsilon$ on $M$, and furthermore the inequality $\lip \widetilde{f}\leq (1+\epsilon)\lip f$ holds. 
It is standard then that for any $\widetilde{r}$ (see for example \cite[Theorem 1]{smoothing}) there exists a $C^{\infty}$ function $\widehat{f}$ on $U$ satisfying 
\[|\widetilde{f}(x)-\widehat{f}(x)|<\widetilde{\varepsilon}(x) \textrm{ for all } x\in U,\]
with $\lip \widehat{f}\leq\lip \widetilde{f} +\widetilde{r}$.
In particular, we deduce $\lip \widehat{f}\leq \lip f +\epsilon\lip f+\widetilde{r}$. Since $\epsilon$ and $\widetilde{r}$ can be chosen to be arbitrarily small, restricting $\widehat{f}$ to $M$ yields the result.
\end{proof}

Coming back to Lemma~\ref{L:exis_tube}, an application of Lemma~\ref{L:approx_lip} with $f=\frac{1}{2}\delta$ and $\varepsilon=\frac{1}{4}\delta$ shows that there is no loss of generality in assuming that $\delta$ is $C^\infty$-smooth and non-expansive. 
For ease of reference, we now record a version of Lemma~\ref{L:exis_tube}, where we impose a number of open conditions on the tubular neighborhood, which will be important for our latter development.

\begin{corollary} \label{L:tubnhd}
Consider a $C^{p+1}$ manifold $M\subset\R^n$, continuous functions $\varepsilon:\R^n\to(0,\infty)$ and $f\colon\R^n\to\R$, a neighborhood $V$ of $M$, and the closed set $\Gamma:=(\cl M)\setminus M$. Then there exists a tubular neighborhood
\[
U_\delta=\left\{y+v\in\R^n: y\in M,\ v\in N_yM,\ |v|<\delta(y)\right\},
\]
with $U_{\delta}\subset V\cap \Gamma^{c}$, and
satisfying
\newcounter{saveenum}
\begin{enumerate}
\item The width function $\delta\colon M\to (0,\infty)$ is $C^{\infty}$-smooth and non-expansive,
\item The metric projection $P_M$ is $C^p$-smooth on $U_{\delta}$,
\setcounter{saveenum}{\value{enumi}}
\end{enumerate}
and for each $x\in U_\delta$ we have 
\begin{enumerate}
\setcounter{enumi}{\value{saveenum}}
\item $|f(x)-f\circ P_M(x)| < \varepsilon(x)$,
\item $|x-P_M(x)| < d(x, \Gamma)^2$,
\item $\| \mathrm D P_M(x) - \mathrm DP_M(y)\| < d(y, \Gamma)$, where $y=P_M(x)$,
\item $\|\mathrm D P_M(x)\| < 2$.
\end{enumerate}
\end{corollary}

\subsection{The interpolation function} \label{S:if}

Given a tubular neighborhood $U_\delta$ as in Corollary~\ref{L:tubnhd} (constructed based on a $C^{p+1}$ manifold $M$, a neighborhood $V$, functions $\varepsilon$ and $f$, and the closed set $\Gamma=(\cl M)\setminus M$), define a function $\phi:\R^n\to \R$ by setting
\[
\phi(x) =
\left\{
\begin{array}{ll}
0  & x \notin U_\delta \\[3mm]
\psi\left( \dfrac{|x - P_M(x)|} {\delta\circ P_M(x)} \right)  &  x\in U_\delta,  
\end{array}
\right.
\]
where $\psi:[0,\infty)\to[0,1]$ is a $C^{\infty}$-smooth cut-off function that is $1$ on $[0,1/4]$, vanishes on $[3/4,\infty)$, and whose first derivative is bounded by $7/3$ in absolute value. See Figure~\ref{F:1} for an illustration. Consider also the ``annulus'' around~$M$, defined by
\[
U_0 = \left\{x\in U_\delta: \frac{1}{4} \delta\circ P_M(x)<|x-P_M(x)|< \frac{3}{4} \delta\circ P_M(x) \right\}.
\]
We refer to the function $\phi$ as the \emph{interpolation function} associated with $U_\delta$, and its decisive properties are collected in the following lemma.
\smallskip

\begin{figure}[htbp]
\begin{center}
\includegraphics[width=90mm]{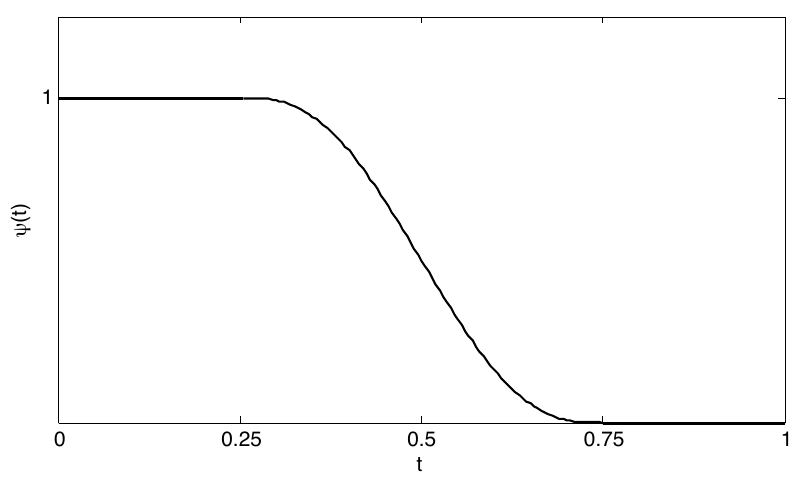}
\caption{The function $\psi$.}\label{F:1}
\label{default}
\end{center}
\end{figure}

\begin{lemma}[Properties of the interpolation function] \label{L:phi1}
The interpolation function $\phi$ satisfies the following properties.
\begin{enumerate}
\item $\phi=1$ on $\{x\in U_\delta: |x-P_M(x)| \le \frac{1}{4} \delta\circ P_M(x) \}$,
\item $\phi=0$ on $\{x\in U_\delta: |x-P_M(x)| \ge \frac{3}{4} \delta\circ P_M(x) \}$ and on $U_\delta^c$,
\item $\phi$ is $C^p$ smooth on $\Gamma^c$, and for all $x\in U_0$, we have
\[
\nabla \phi(x) = t\psi'( t ) \left[ \frac{x-P_M(x)}{|x-P_M(x)|^2} - \frac{\nabla(\delta\circ P_M)(x)}{\delta\circ P_M(x)} \right],
\]
where $t= \frac{|x - P_M(x)|}{\delta\circ P_M(x)}$.
\item For all $x\in U_\delta$ the inequality $|x-P_M(x)|\  |\nabla\phi(x)| \le 8$, holds.
\end{enumerate}
\end{lemma}

\begin{proof}
The first three properties follow directly from the definition, together with a calculation to compute the gradient. The only point that is somewhat delicate here is to verify the equality
\begin{equation}\label{eq:proj_form}
\nabla |\id - P_M|(x) = \frac{x-P_M(x)}{|x-P_M(x)|},
\end{equation}
for $x\in U_{\delta}\setminus M$. To this end, an application of the chain rule yields the equation
\[
\langle \nabla |\id - P_M|(x), u\rangle = \frac{\langle u -  \mathrm DP_M(x)[u], x-P_M(x)\rangle}{|x-P_M(x)|},
\]
for any vector $u\in\R^n$.
This together with the inclusions $\mathrm DP_M(x)[u]\in T_{P_M(x)}M$ and $x-P_M(x)\in N_{P_M(x)}M$, establishes (\ref{eq:proj_form}). 

For the last property, first note that $\nabla\phi=0$ holds outside $U_0$. Hence it is sufficient to consider only points $x\in U_0$. For such points, observe that the inequalities $|t\psi'(t)|\le \frac{7}{4}$ and $|x-P_M(x)|<\delta\circ P_M(x)$ hold. Since $\delta$ is non-expansive, and Corollary~\ref{L:tubnhd} yields the bound $\|\mathrm DP_M\|\le 2$, we readily obtain the inequality $|\nabla(\delta\circ P_M)|\le 2$. The claim now follows from the expression for~$\nabla\phi(x)$.
\end{proof}

\subsection{The induction step} \label{S:it1}
We begin with the following straightforward lemma.
\begin{lemma} \label{L:V_M}
Consider a $C^2$-stratification $\mathcal{A}$ of a set $Q\subset\R^n$.  Then there exists a family of sets $\{V_M\}_{M\in\mathcal{A}}$, where for each stratum $M\in\mathcal A$ the set $V_M$ is a neighborhood of $M$, and the following properties hold: 
\begin{itemize}
\item If we have $L\prec M$, then $L\cap V_M=\emptyset$, 
\item If $L$ and $M$ are incomparable, then $V_L\cap V_M=\emptyset$.
\end{itemize}
\end{lemma}

\begin{proof}
Consider any pair of strata $L,M\in\mathcal A$. If we have $L\prec M$, then  define $V_{L, M}:=(\cl L)^c$. Similarly, if we have $M\prec L$, then define $V_{M,L}:=(\cl M)^c$. If $L$ and $M$ are incomparable, then by the definition of a stratification, there exist disjoint neighborhoods $V_{L,M}$ of $M$ and $V_{M,L}$ of $L$. Observe that for any strata $L,M\in\mathcal A$, the set $V_{L,M}$, if defined, is a neighborhood of $M$ and $V_{M,L}$, if defined, is a neighborhood of $L$. 

Define now for each stratum $M$, the neighborhood $V_M$ to be the intersection of the (finitely many) sets $V_{L,M}$ where either $L\prec M$ or $L$ and $M$ are incomparable. If M is the only stratum, then simply define $V_M:=\R^n$. It is easy to see that this yields a family of sets $\{V_M\}_{M\in\mathcal{A}}$ with the desired properties.
\end{proof}

For the rest of this subsection, we fix the following objects:
\begin{itemize}[leftmargin=0.75cm]
\item $\mathcal A$, a $C^2$-stratification of some closed set $Q\subset\R^n$. We emphasize that $\mathcal A$ is not assumed to satisfy the Whitney condition~(a).
\item $M\in\mathcal A$, a stratum.
\item $\Gamma = (\textrm{cl } M) \setminus M$, a closed set, possibly empty. Since $Q$ is closed, the set $\Gamma$ is the union of all strata $L\in \mathcal A$ satisfying $L\prec M$.
\item $\varepsilon:\R^n\to(0,\infty)$, a continuous function.
\item $f:\R^n\to\R$, a locally Lipschitz continuous and differentiable function. We emphasize that the gradient $\nabla f$ is not assumed to be continuous.
\item $U_\delta$, a tubular neighborhood of $M$. It is obtained by applying Corollary~\ref{L:tubnhd} with the functions $\varepsilon$, $f$, and $\Gamma$ as above, and $V$ taken to be the neighborhood $V_M$ of $M$ given by Lemma~\ref{L:V_M}.
\item $\phi$, the interpolation function corresponding to $U_\delta$.
\end{itemize}

Since $\Gamma$ is a union of pairwise disjoint submanifolds in $\mathcal A$, we can unambiguously let $T_x\Gamma$ and $N_x\Gamma$, for $x\in\Gamma$, denote the tangent and normal spaces of the submanifold where $x$ lies. Note that we have so far made no assumptions on the geometry of $Q$ beyond the existence of a $C^2$-stratification.

\begin{proposition}[differentiable approximation] \label{P:grad1} {\hfill \\}
Assume that the inclusion $\nabla f(x) \in T_x\Gamma$ holds for every $x\in\Gamma$, and define
\[
g = \phi  f \circ P_M + (1-\phi) f.
\]
Then $g$ satisfies the following properties.
\begin{itemize}[leftmargin=0.75cm]
\item $g$ is locally Lipschitz continuous and differentiable.
\item $\nabla g(x) = \nabla f(x)$ for all $x\in\Gamma$, and $\nabla g(x) \in T_xM$ for all $x\in M$.
\item $|f(x)-g(x)| \le \varepsilon(x)$ for all $x\in\R^n$.
\item if $f$ is Lipschitz continuous, then so is $g$, and the inequality $\lip g \leq 11 ~\lip f$ holds.
\end{itemize}
If $\mathcal A$ satisfies the Whitney condition~(a) and $f$ is $C^1$-smooth, then $g$ is $C^1$-smooth as well. 
\end{proposition}
\begin{proof}
It is convenient to split the proof into a number of steps.

{\it Step 1 (differentiability).} We show that $g$ is differentiable at each $\bar{x}\in \Gamma$ with $\nabla g(\bar{x}) = \nabla f(\bar{x})$. This is sufficient to establish the claim, since clearly $g$ is differentiable on $\Gamma^c$. To this end, let $v_k$ be any sequence of unit vectors in $\R^n$ converging to some vector $\bar v$, and let $t_k$ be a sequence of positive reals converging to $0$. We set $x_k:=\bar{x}+t_kv_k$, and assuming without loss of generality $x_k\in U_\delta$ for each $k$, we define $y_k:=P_M(x_k)$. Observe $g(\bar{x})=f(\bar{x})$, and consequently we have
\begin{equation} \label{eq:grad1}
\frac{g(x_k) - g(\bar{x})}{t_k} = \frac{f(x_k)-f(\bar{x})}{t_k} + \phi(x_k) \frac{f(y_k) - f(x_k)}{t_k}.
\end{equation}
The first term on the right-hand-side converges to $\langle \nabla f(\bar{x}), \bar v\rangle$. For the second term, the Mean Value Theorem applied to the function $t\mapsto f(x_k+t(y_k-x_k))$ shows that there exists a vector $z_k$, lying on the line segment joining $x_k$ and $y_k$, satisfying $f(y_k)-f(x_k)=\langle \nabla f(z_k),y_k-x_k\rangle$. Hence we obtain
\[
\frac{|f(y_k) - f(x_k)|}{t_k} \le  |\nabla f(z_k)| \frac{|y_k-x_k|}{t_k} <  |\nabla f(z_k)| t_k,
\]
where we used the inequality $|y_k-x_k| < d(x_k,\Gamma)^2 \le |x_k-\bar{x}|^2 = t_k^2$, which is a direct consequence of Corollary~\ref{L:tubnhd}. Since $\nabla f$ is locally bounded due to local Lipschitzness of $f$, we see that the second term on the right-hand-side of (\ref{eq:grad1}) converges to zero, thereby establishing the claim.

{\it Step 2 (closeness, tangency condition, Lipschitz modulus).} The closeness of~$g$ to~$f$ follows from the simple calculation
\[
|f(x)-g(x)| = \phi(x) |f(x) - f\circ P_M(x)| \le \varepsilon(x),
\]
where the last inequality uses Corollary~\ref{L:tubnhd}.
Next, we verify the tangency condition. For any $x\in M$ and $v\in N_xM$ we have $g(x+t v) = f(x)$ for all $t$ close to zero, since the equality $\phi=1$ holds in a whole neighborhood of $x$. It follows that the directional derivative of $g$ in the direction $v$ vanishes at $x$, and hence we deduce $\nabla g(x)\in T_xM$ as desired. Finally, we prove that~$g$ is locally Lipschitz continuous and derive a bound on its Lipschitz modulus. For $x\in\Gamma^c$, we have
\[
\nabla g(x) = (1-\phi(x))\nabla f(x) +  \nabla \phi(x) ( f\circ P_M (x)- f(x)) +  \phi(x) \nabla (f\circ P_M)(x),
\]
and hence
\[
|\nabla g(x)| \le |\nabla f(x)| + |\nabla\phi(x)| \ |f\circ P_M(x)-f(x)| + \phi(x)|\nabla(f\circ P_M)(x)| .
\]
For any $x\in U_\delta$ and $w\in\R^n$ we have
\[
\langle \nabla (f\circ P_M)(x), w\rangle = \langle \nabla f (P_M(x)), {\rm D}P_M(x)w \rangle,
\]
so that, due to Corollary~\ref{L:tubnhd}(6), we have $|\nabla(f\circ P_M)(x)|\le 2|\nabla f(P_M(x))|$.
Moreover, by the Mean Value Theorem there exists a point $z$ on the line segment joining $x$ and $P_M(x)$, satisfying $f(x)-f(P_M(x))=\langle \nabla f(z),x-P_M(x)\rangle$. Consequently we deduce 
\begin{align*}
|\nabla g(x)|
&\le |\nabla f(x)| + 2\phi(x) |\nabla f(P_M(x))| + |\nabla \phi(x)|\ |x-P_M(x)|\ |\nabla f(z)| \\
&\le
\begin{cases}
|\nabla f(x)|, &  x \notin U_\delta \\[2mm]
|\nabla f(x)| + 2|\nabla f(P_M(x))| + 8|\nabla f(z)|, & x \in U_\delta,
\end{cases}
\end{align*}
where the second inequality used Lemma~\ref{L:phi1}(4). Given any compact set $K\subset\R^n$, the points $P_M(x)$ and $z$ remain inside some other compact set $K'$ as $x$ varies over $K\cap U_\delta$. This yields the inequality
\[
\sup_{x\in K}|\nabla g(x)| \le 11 \sup_{x\in K'} |\nabla f(x)|.
\]
It easily follows that $g$ is locally Lipschitz continuous. If $f$ is globally Lipschitz continuous, then so is $g$, and we have $\lip g \le 11 ~\lip f$, as claimed.

Assume now that $\mathcal A$ satisfies the Whitney condition~(a) and $f$ is $C^1$-smooth.

{\it Step 3 (continuity of the gradient).} We now show that the gradient mapping $\nabla g$ is continuous. It is enough to check continuity at each point $\bar{x}\in\Gamma$. For $x\in\Gamma^c$ we have
\begin{equation} \label{eq:gradg}
\nabla g(x) - \nabla f(x) =  \nabla \phi(x) ( f\circ P_M (x)- f(x)) +  \phi(x) (\nabla (f\circ P_M)(x) - \nabla f(x)).
\end{equation}
Since $\nabla f(\bar x)=\nabla g(\bar x)$ holds and $\nabla f$ is continuous, it suffices to show that the right-hand-side converges to zero as $x$ tends to $\bar{x}$ with $x\in U_\delta$. We do this term by term. To deal with the first term, we define $y:=P_M(x)$ and apply the Mean Value Theorem to obtain a vector $z$ on the line segment joining $x$ and $y$, satisfying $f(y)-f(x)=\langle \nabla f(z),y-x\rangle$. Assuming without loss of generality $x\ne y$, we obtain
\[
|f(y)-f(x)| = |y-x| \langle \nabla f(z), v \rangle,
\]
where $v:=\frac{y-x}{|y-x|}\in N_yM$. Whenever $v$ converges to some vector $\bar v$, we have $\bar v\in N_{\bar{x}}\Gamma$ by the Whitney condition~(a). Consequently along any convergent subsequence, we obtain $\langle \nabla f(z), v \rangle \to \langle \nabla f(\bar{x}),\bar v\rangle=0$, since $\nabla f$ is continuous and $\nabla f(\bar{x})$ lies in $T_{\bar{x}}\Gamma$. Finally recalling from Lemma~\ref{L:phi1} that the quantity $|\nabla \phi(x)|\,|y-x|$ is bounded, we deduce that the first term on the right-hand-side of~\eqref{eq:gradg} tends to zero.  

For the second term, it suffices to show $|\nabla(f\circ P_M)(x)-\nabla f(x)|\to 0$ as $x\to\bar{x}$ with $x\in U_\delta$. To this end, apply the equality
\[
\langle \nabla(f\circ P_M)(x), w\rangle = \langle \nabla f(y), \textrm{D}P_M(x)w \rangle,
\]
valid for any $w\in\R^n$, to obtain
\begin{align*}
\langle \nabla(f\circ P_M)(x)-\nabla f(x), w\rangle
&= \langle \nabla f(x), \textrm{D}P_M(y)w - w \rangle \\[2mm]
&\qquad + \langle \nabla f(y) - \nabla f(x), \textrm{D}P_M(y)w \rangle \\[2mm]
&\qquad\qquad + \langle \nabla f(y), \textrm{D}P_M(x)w - \textrm{D}P_M(y)w \rangle.
\end{align*}
Observe that the inclusion $\textrm{D}P_M(y)w - w\in N_yM$ holds, and so any limit point of such vectors lies in $N_{\bar{x}}\Gamma$. Hence along every convergent subsequence, the first term vanishes. The second and third terms vanish due to the continuity of $\nabla f$, boundedness of $\|\textrm{D}P_M(y)\|$, and the inequality $\| \textrm{D} P_M(x) - \textrm{D}P_M(y)\| < d(y, \Gamma)$ established in Corollary~\ref{L:tubnhd}. This concludes the proof of Step~3.
\end{proof}

If we assume that the stratification $\mathcal A$ is normally flat, stronger approximation results can be obtained. First we record the following observation, which is a direct consequence of the finiteness of the number of strata and the defining property of normally flat stratifications.

\begin{lemma} \label{L:A1}
Consider a normally flat $C^p$-stratification $\mathcal A$ of a set $Q\subset\R^n$ and a stratum $M$. Then there exists a neighborhood $\widehat U$ of~$M$ such that for each stratum $L\in\mathcal A$ with $L\prec M$, there is a neighborhood $\widehat V$ of $L$ with $P_L(x)=P_L\circ P_M(x)$ for all $x\in \widehat U\cap\widehat V$.
\end{lemma}

\begin{proposition}[Arbitrarily smooth approximation] \label{P:Cp}
Suppose that the stratification $\mathcal A$ is normally flat, and that $U_\delta$ is contained in $\widehat U$ from Lemma~\ref{L:A1}. Assume that each stratum $L\in\mathcal A$ with $L\prec M$ has a neighborhood $V$ with
\[
f(x) = f\circ P_L(x) \quad \textrm{for all} \quad x\in V.
\]
Define a function $g = \phi  f \circ P_M + (1-\phi) f$. Then each stratum $L\in\mathcal A$ with $L\subset \cl M$ has a neighborhood $W$ so that the equality
\[
g(x) = g\circ P_L(x) \quad \textrm{holds for all} \quad x\in W.
\]
\end{proposition}

\begin{proof}
We start with the case $L=M$. To this end, let $W$ be the subset of $U_\delta$ where the equality $\phi=1$ holds. Then $W$ is a neighborhood of $M$, and for any point $x\in W$, we have
\[
g(x) = f\circ P_M(x) = f\circ P_M \circ P_M(x) = g \circ P_M(x),
\]
as required. Consider now a stratum $L\in\mathcal A$ with $L\prec M$. By hypothesis there is a neighborhood $V$ of $L$ in which the equality $f=f\circ P_L$ holds. Now pick a new neighborhood $W$ of $L$ satisfying
\begin{itemize}[leftmargin=0.75cm]
\item $W\subset V$,
\item $P_M(x)\in V$ for all $x\in W\cap U_\delta$,
\item $P_L\circ P_M(x) = P_L(x)$ for all $x\in W\cap U_\delta$.
\end{itemize}
To see that such a neighborhood can be found, first note that it is easy to find $W$ satisfying the first two properties, thanks to the continuity of $P_M$ on $U_{\delta}$: simply intersect $V$ with a sufficiently small open set containing $L$. Then replace $W$ by the smaller set $W\cap\widehat V$, where $\widehat V$ is as in Lemma~\ref{L:A1}. The resulting set, which we again denote by $W$, satisfies all three properties.

We now verify the equality $g=g\circ P_L$ on $W$. First consider a point $x\in W \setminus U_\delta$. Recall that we have $\phi=0$ outside $U_\delta$ and $f=f\circ P_L$ on $W$. Hence we obtain
\[
g(x) = f(x) = f\circ P_L(x) = g\circ P_L(x).
\]
Now consider instead a point $x\in W\cap U_\delta$. The inclusion $P_M(x)\in V$, the equality $f=f\circ P_L$ on $V$, and the normal flatness condition yield
\[
f\circ P_M(x) = f\circ P_L\circ P_M(x) = f \circ P_L(x).
\]
Consequently, we deduce
\[
g(x) = \phi(x) f\circ P_M(x) + (1-\phi(x))f(x) = f\circ P_L(x).
\]
Finally note that the right-hand-side equals $g\circ P_L(x)$, since we have $\phi(P_L(x))=0$. This establishes the claim.
\end{proof}

We end this subsection by recording the following remark.
\begin{remark} \label{R:Cp}
{\rm
In the notation of Proposition~\ref{P:Cp}, observe that since the equality $\phi=0$ holds on $\Gamma=(\cl M)\setminus M$, we have $g=f$ there. Consequently for any stratum~$L$ contained in $\Gamma$, there exists a neighborhood $V$ of $L$, so that for all $x\in V$ we have
\[
g(x) = g\circ P_L(x) = f\circ P_L(x) = f(x).
\]
Hence if $f$ is $C^p$-smooth and the restriction of $\phi$ to $\Gamma^c$ is $C^p$-smooth, then $g$ is $C^p$-smooth as well. }
\end{remark}

\subsection{Proofs of main results} \label{S:pf}

Given the previous developments, the proofs of our main results are entirely straightforward.

\begin{proof}[Proof of Theorem~\ref{thm:main}]
We first establish some notation. Given a differentiable function $h$ we say that $h$ is {\em tangential to $M\in\mathcal A$} if the inclusion $\nabla h(x)\in T_xM$ holds for all $x\in M$. We let $\mathcal T_h$ be the collection of $M\in\mathcal A$ such that $h$ is tangential to~$M$, and we define $m$ to be the depth of $\mathcal{A}$.

The proof proceeds by induction on the partial order~$\prec$. In preparation for this, we set $f_0:=f$ if $f$ is already differentiable and locally Lipschitz continuous. Otherwise we let $f_0$ be a $C^1$-smooth function that differs from $f$ by at most $\varepsilon/(m+1)$, and whose Lipschitz modulus (when $f$ is Lipschitz continuous) is at most 11 times that of~$f$.

For the induction step, suppose a differentiable locally Lipschitz continuous function~$h$ is given, pick $M\in\mathcal A$ with $M\notin\mathcal T_{h}$, and assume we have $\{L\in\mathcal A: L\prec M\}\subset \mathcal T_{h}$ (this is the induction assumption). Proposition~\ref{P:grad1} then gives a new differentiable locally Lipschitz continuous function~$h_1$ such that $\mathcal T_{h_1}=\mathcal T_{h}\cup\{M\}$ and $|h-h_1|<\varepsilon/(m+1)$ hold. 

Observe that for each minimal stratum $M$ the family $\{L\in\mathcal A:L\prec M\}$ is empty, so the induction assumption is vacuously true for $f_0$. It then follows by induction on the partial order~$\prec$ that we may find a differentiable locally Lipschitz continuous function~$g$ such that~$\mathcal T_g=\mathcal A$.

At each application of Proposition~\ref{P:grad1}, the modified function differs from the previous one by at most $\varepsilon/(m+1)$. However, the modifications associated with incomparable strata do not overlap. This is because the tubular neighborhood $U_\delta$ corresponding to $M\in\mathcal A$ is contained in $V_M$ from Lemma~\ref{L:V_M}. Hence the total error in the end is only $\varepsilon m/(m+1)$. Together with a possible initial error of $\varepsilon/(m+1)$ if~$f$ is not differentiable and locally Lipschitz continuous, this results in the claimed bound $|f-g|<\varepsilon$. Similarly, we lose a factor of $11^{m+1}$ in the Lipschitz modulus. To ensure that the support of $g$ is contained in $V_1$, we simply ensure that the tubular neighborhood $U_\delta$ is contained in $V_1$ at each application of Proposition~\ref{P:grad1}. The same can be done with $V_2$ whenever the proposition is applied with a stratum $M$ such that $\dim M<n$. For $M$ with $\dim M=n$, the tubular neighborhood $U_\delta$ coincides with $M$, and we have $P_M=\id$ on $M$. Hence in this case Proposition~\ref{P:grad1} yields a function that is identical to the previous one. We deduce that if $f$ was already $C^1$-smooth, so that no initial smoothing took place (i.e., $f_0=f$), then $f=g$ holds outside~$V_2$. Finally, in case $\mathcal A$ satisfies the Whitney condition~(a), the $C^1$-smoothness of $f_0$ is preserved after each application of Proposition~\ref{P:grad1}, and this yields the final assertion of Theorem~\ref{thm:main}.
\end{proof}

\begin{proof}[Proof of Theorem~\ref{T:smooth}]
The result follows exactly as in the proof of Theorem~\ref{thm:main}, except that in each step we additionally apply Proposition~\ref{P:Cp}. This can be done, provided we shrink $U_\delta$ if necessary. Note that the smoothness carries over in each step by Remark~\ref{R:Cp}. The function $g$ obtained in the end has the property that every stratum $M\in\mathcal A$ has a neighborhood on which $g=g\circ P_M$ holds.
\end{proof}

\section{Normally flat stratifications} \label{S:A1}
In this section we discuss important examples of sets admitting normally flat stratifications. First, however, we observe that this notion is strictly stronger than the Whitney condition~(a).

\begin{proposition}[Normally flat stratifications are Whitney (a)-regular]\label{P:A1Wh} {\hfill \\ }
If a $C^p$-stratification of a set $Q\subset\R^n$ is normally flat, then for any strata $L$ and $M$ with $L\prec M$ there exists a neighborhood $V$ of $L$ so that the inclusion 
\[N_{x}M\subset N_{P_L(x)}L \quad\textrm{ holds for all } x\in M\cap V.\]
Consequently normally flat $C^p$-stratifications satisfy the Whitney condition (a).
\end{proposition}
\begin{proof}
Since the stratification is normally flat, there exist neighborhoods $V$ of $L$ and $U$ of $M$ so that the equality
\[
P_L(x) = P_L \circ P_M(x) \quad \textrm{ holds for all}\quad x\in V\cap U.
\]
Consider a point $x \in M\cap V$, and let $v\in N_{x}M$ be an arbitrary normal vector. Define $y=x+t v$ for some $t>0$ that is sufficiently small to guarantee the inclusion $y\in U$. Then we have $x=P_M(y)$, and consequently $P_L(x)=P_L\circ P_M(y)=P_L(y)$. We deduce
\[
t v = y - x = \left( y - P_L(y) \right) + \left( P_L(x) - x\right)\in N_{P_L(x)}L,
\]
as we had to show. The claim that normally flat $C^p$-stratifications satisfy the Whitney condition~(a) is now immediate.  
\end{proof}

The following simple example confirms that the normal flatness condition is indeed strictly stronger than the Whitney condition~(a).

\begin{example}[Whitney (a)-regular stratification that is not normally flat]{\hfill \\}
Consider the set \[Q=\{(x,y,z)\in\R^3: z=xy, x\ge 0\},\] together with the stratification $\mathcal A=\{M_1,M_2\}$, where the strata are defined by 
\[M_1=\{0\}\times\R\times\{0\} \quad\textrm{ and }\quad M_2=\{(x,y,z)\in Q: x> 0\}.\] Clearly $\mathcal{A}$ is a Whitney (a)-regular stratification. On the other hand, this stratification is not normally flat. To see this, first note the equivalence \[P_{M_1}(x,y,z) = (0,0,0)\quad \Longleftrightarrow\quad y=0.\] On the other hand, $P_{M_2}(x,0,z)$ is a singleton and has a non-zero $y$-coordinate whenever we have $x>0$ and $z\ne 0$.
\end{example}

To record our first examples of normally flat stratifications, we need the following standard result.
\begin{lemma}[Ordered affine subspaces]\label{lem:aff}
Consider affine subspaces $L$ and $M$ in~$\R^n$, satisfying the inclusion $L\subset M$. Then  the equality
\[
P_L(x) = P_L \circ P_M(x) \quad \textrm{holds for all}\quad x\in \R^n.
\]
\end{lemma}
\begin{proof}
Consider an arbitrary point $x\in \R^n$. Observe that for any point $y\in L$, we have \[x-P_M(x)\in N_{P_M(x)}M \quad\textrm{ and }\quad y-P_M(x)\in T_{P_M(x)}M.\] Consequently any two such vectors are orthogonal and we obtain
\begin{align*}
|x-y|^2=|x-P_M(x)|^2+|P_M(x)-y|^2&\geq |x-P_M(x)|^2+|P_M(x)-P_L(P_M(x))|^2\\
&=|x-P_L(P_M(x))|^2.
\end{align*}
Since $y\in L$ is arbitrary, we deduce $P_L(x) = P_L \circ P_M(x)$, as claimed.
\end{proof}

We denote the {\em affine hull} of any convex set $Q\subset\R^n$ by $\aff Q$. 
\begin{proposition}[Polyhedral stratifications are normally flat] \label{P:polflat}
Consider a stratification of a set $Q\subset\R^n$, where each stratum is an open polyhedron. Then the stratification is normally flat. 
\end{proposition}
\begin{proof}
Consider strata $L$ and $M$, with $L\prec M$. Clearly there exists a neighborhood $V$ of $M$ and $U$ of $L$ satisfying 
\[P_L=P_{\scriptsize{\aff L}} \textrm{ on }V \quad\textrm{ and } \quad P_M=P_{\scriptsize{\aff M}} \textrm{ on } U.\] Furthermore observe that the inclusion, $\aff L\subset \aff (\cl M)=\aff M$, holds. The result is now immediate from Lemma~\ref{lem:aff}.
\end{proof}

It turns out that normally flat stratifications satisfy the so-called {\em Transfer Principle}, which we will describe below. This realization will show that many common subsets of matrices admit normally flat stratifications. We first introduce some notation.

\begin{itemize}[leftmargin=0.75cm]
\item $\mathbf M^{n\times m}$ is the Euclidean space of all $n\times m$ real matrices, endowed with the trace inner product $\langle X,Y\rangle:=\tr(X^\top Y)$ and Frobenius norm $\|X\|_F:=\langle X^\top X\rangle^{1/2}$. For notational convenience, throughout we will assume $n\leq m$.
\item $\mathbf S^n$ is the subspace (when $m=n$) of all symmetric matrices, endowed with the trace inner product and the Frobenius norm.
\item $\R^n_\ge$ is the set $\{(x_1,\ldots,x_n)\in\R^n : x_1\ge x_2 \ge \dots \ge x_n\}$, and $\R^n_{+,\ge} := \R^n_\ge \cap \R^n_+$.
\item $\lambda:\mathbf S^n\to \R^n$ is the map taking $X$ to its vector $(\lambda_1(X),\ldots,\lambda_n(X))$ of (real) eigenvalues, in decreasing order.
\item $\sigma:\mathbf M^{n\times m}\to \R^n$ is the map taking $X$ to its vector $(\sigma_1(X),\ldots,\sigma_n(X))$ of singular values, in decreasing order. (Recall throughout we are assuming $n\leq m$.)
\item $\Diag x\in{\bf M}^{n\times m}$, for a vector $x\in\R^n$, is a matrix that has entries all zero, except for its principle diagonal, which contains the entries of $x$.
\end{itemize}

For notational convenience, sets $Q\subset\R^n$ that are invariant under any permutation of coordinates will be called {\em permutation-invariant}, while sets that are invariant under any coordinate-wise change of sign and permutation of coordinates will be called {\em absolutely permutation-invariant}.

Let $O(n)$ be the group of $n\times n$ orthogonal matrices. We can define an action of $O(n)$ on the space of symmetric matrices $\mathbf S^n$ by declaring
\[U.X=U X U^\top \textrm{ for all } U\in O(n) \textrm{ and } X \textrm{ in }\mathbf S^n.\] We say that a subset of $\mathbf S^n$ is {\em spectral} if it is invariant under the action of $O(n)$. Equivalently, a subset of $\mathbf S^n$ is spectral if and only if it can be represented as  
$\lambda^{-1}(Q)$, for some permutation-invariant set $Q\subset\R^n$. Due to this invariance of $Q$, the spectral set can be written simply as the union of orbits
\[\lambda^{-1}(Q)=\bigcup_{x\in Q} O(n).(\Diag x)\]

Similarly, we can consider the Cartesian product $O(n)\times O(m)$, which we denote by  $O(n,m)$, and its action on the space $\mathbf M^{n\times m}$ defined by 
\[(U,V).X=U X V^\top \textrm{ for all } (U,V)\in O(n,m) \textrm{ and } X \textrm{ in }\mathbf M^{n\times m}.\] We say that subset of $\mathbf M^{n\times m}$ is {\em spectral} if it is invariant under the action of $O(n,m)$. Equivalently, a subset of $\mathbf M^{n\times m}$ is spectral if and only if it has the form  
$\sigma^{-1}(Q)$,
for some absolutely permutation-invariant set $Q\subset\R^n$. In this situation, the spectral set is simply the union of orbits,
\[Q=\bigcup_{\scriptsize{x\in Q}} O(n,m).(\Diag x)\]

The mappings $\sigma$ and $\lambda$ have nice
geometric properties, but are very badly behaved as far as, for example, differentiability is concerned. However such difficulties are alleviated by the invariance assumptions on $Q$. The {\em Transfer Principle} asserts that many geometric (or more generally variational analytic) properties of the sets $Q$ are inherited by the spectral sets $\sigma^{-1}(Q)$ and $\lambda^{-1}(Q)$. The collection of properties known to satisfy this principle is impressive: convexity \cite{lag}, prox-regularity \cite{prox}, Clarke-regularity \cite{send,lag}, 
smoothness~\cite{lag,der,high_order,man,man2, spec_id} and partial smoothness~\cite{spec_id}. In this section, we will add the existence of normally flat stratifications to this list.

The following crucial lemma asserts that projections interact well with the singular value map (or eigenvalue map in the symmetric case). This result may be found in \cite[Theorem A.1, Theorem A.2]{alt_proj}, and in \cite[Proposition 2.3]{prox} for the symmetric case. For completeness, we record a short proof, adapted from \cite[Proposition 2.3]{prox}.

\begin{lemma}[Projection onto spectral sets] \label{L:PA1sp} {\hfill \\}
Let $Q\subset\R^n$ be an absolutely permutation-invariant set and consider a matrix $X\in\mathbf M^{n\times m}$ with singular value decomposition $X=U\Diag \sigma(X)~V^\top$. Then the inclusion 
\begin{equation}\label{eq:proj}
P_{\scriptsize{\sigma^{-1}(Q)}}(X) \supset \{U\Diag z~V^\top: z\in P_{Q}(\sigma(X))\} \quad\textrm{ holds}.
\end{equation}
Similarly if $Q\subset\R^n$ is only permutation-invariant and $X\in \mathbf S^n$ has an eigenvalue decomposition $X=U\Diag \lambda(X)~U^\top$, then the analogous inclusion 
\[
P_{\lambda^{-1}(Q)}(X) \supset \{U\Diag z~U^\top: z\in P_{Q}(\lambda(X))\} \quad\textrm{ holds}.
\]
\end{lemma}

\begin{proof}
The proof relies on the fact that the singular-value map $\sigma$ is non-expansive, see Theorem~7.4.51 in~\cite{Horn/Johnson:1985}. Then for arbitrary $Y\in Q$ and $z\in P_{Q}(\sigma(X))$, we obtain
\[
\|X-Y\|_F \ge |\sigma(X)-\sigma(Y)| \ge |\sigma(X) - z|,
\]
where the second inequality follows from the equivalence $\sigma(Y)\in Q\Longleftrightarrow Y\in \sigma^{-1}(Q)$. On the other hand, the matrix $Y=U\Diag z~V^\top$ achieves equality and lies in $\sigma^{-1}(Q)$, and this yields the claim. The symmetric case follows in the same way since the eigenvalue map $\lambda$ is also non-expansive.
\end{proof}

\begin{remark}{\rm
Lemma~\ref{L:PA1sp} shows that in the particular case when $P_{\scriptsize{\sigma^{-1}(Q)}}(X)$ is a singleton, so is $P_{Q}(\sigma(X))$, and we have equality in 
(\ref{eq:proj}). We furthermore have the appealing formula
\[
\sigma \circ P_{\sigma^{-1}( Q)} = P_{ Q}\circ \sigma \qquad (\textrm{or } \lambda\circ P_{\lambda^{-1}( Q)} = P_{ Q}\circ\lambda),
\]
which captures the gist of the result.}
\end{remark}

The proof of the following simple lemma may be found in \cite[Lemmas A.1, A.2]{alt_proj}, though the statement of the result is somewhat different.

\begin{lemma}[Permutations of projected points]\label{lem:permute_proj}
Consider an absolutely permutation invariant set $Q\subset\R^n$ and a point $x\in\R^{n}_{+,\geq}$. Then for any point $y$ lying in $P_Q(x)$, there exists a signed permutation matrix $A$ on $\R^n$ so that $Ay$ lies in $\R^{n}_{+,\geq}\cap P_Q(x)$. The analogous statement holds in the symmetric case.
\end{lemma}

\begin{theorem}[Lifts of stratifications] \label{P:A1sp}
Consider a partition $\mathcal A$ of a set $Q\subset \R^n$ into finitely many $C^p$ manifolds that are absolutely permutation-invariant. 
Then if $\mathcal A$ is a $C^p$-stratification of $Q$, the family
\[
\sigma^{-1} (\mathcal A) := \left\{\sigma^{-1}(M) : M \in \mathcal A\right\},
\]
is a $C^p$-stratification of $\sigma^{-1}(Q)$.
The analogous statement holds for both Whitney~(a)-regular and normally flat $C^p$-stratifications. The case of symmetric matrices is analogous as well.
\end{theorem}

\begin{proof}
First, we note that by \cite[Theorem~5.1]{spec_id}, each set $\sigma^{-1}(M)$ for $M\in \mathcal{A}$ is a $C^p$ manifold.
Throughout the proof, we let $L$ and $M$ be arbitrary strata in $\mathcal A$. We first claim that the implication 
\begin{equation}\label{eqn:imp}
L\subset \cl M\Longrightarrow \sigma^{-1}(L)\subset\cl \sigma^{-1}(M),
\end{equation}
holds. Indeed consider a matrix $X\in\sigma^{-1}(L)$ and a singular value decomposition $X=U \Diag \sigma(X)~V^T$. Then there exists a sequence $x_i\to \sigma(X)$ in $M$. Observe that since $M$ is absolutely permutation-invariant, the matrices $X_i:=U\Diag x_i~V^T$ lie in $\sigma^{-1}(M)$ and converge to $X$. This establishes the validity of the implication.

Assume now that $\mathcal A$ is a $C^p$-stratification of $Q$. Clearly $\sigma^{-1}(\mathcal A)$ is a partition of $\sigma^{-1}(Q)$. Now suppose $\sigma^{-1}(L)\cap \cl \sigma^{-1}(M)\neq\emptyset$. Then by continuity of $\sigma$, we have $L\cap \cl M\neq\emptyset$. We deduce that the inclusion $L\subset \cl M$ holds. Then using $(\ref{eqn:imp})$, we obtain $\sigma^{-1}(L)\subset\cl \sigma^{-1}(M)$, thus verifying that $\sigma^{-1} (\mathcal A)$ is a $C^p$-stratification of 
$\sigma^{-1}(Q)$.

Assume now that in addition, $\mathcal A$ is a Whitney (a)-regular $C^p$-stratification of $Q$. Consider a point $X\in M$. Then by \cite[Theorem 7.1]{send}, we have the formula
\[N_{\scriptsize{X}} (\sigma^{-1}(M))=\{U(\Diag N_{\sigma(X)}M) V^\top: U\Diag \sigma(X) V^\top=X\}.\]
Verification of the fact that $\sigma^{-1} (\mathcal A)$ is a Whitney (a)-regular $C^p$-stratification of 
$\sigma^{-1}(Q)$ is now trivial.

Assume now that $\mathcal A$ is a normally flat $C^p$-stratification. We will show that $\sigma^{-1}(\mathcal A)$ is a normally flat $C^p$-stratification of $\sigma^{-1}(Q)$.
To this end, suppose that the inclusion $\sigma^{-1}(L)\subset\cl \sigma^{-1}(M)$ holds. Then there exist neighborhoods $V$ of $L$ and $U$ of $M$, with $P_{L}=P_{L}\circ P_{M}$ on $V\cap U$. Define the neighborhoods
\[
\widetilde{U} = \sigma^{-1}(U) \qquad \textrm{and} \qquad \widetilde{V} = \sigma^{-1}(V),
\]
of $M$ and $L$, respectively. Shrinking $\widetilde{U}$ and $\widetilde{V}$, we may suppose that the maps $P_L$, $P_M$, and the composition $P_L\circ P_M$ are all well-defined and single-valued on $\widetilde{V}\cap \widetilde{U}$. Fix a matrix $X\in \widetilde{U}\cap \widetilde{V}$, with singular value decomposition $X=U\Diag\sigma(X)~V^\top$. Applying Lemma~\ref{L:PA1sp} and Lemma~\ref{lem:permute_proj} successively, we deduce
\begin{eqnarray*}
P_{\sigma^{-1}(L)}\circ P_{\sigma^{-1}(M)}(X) &=&  P_{\sigma^{-1}(L)}( U \Diag P_{M}(\sigma(X))~V^\top ) \\
&=& U \Diag [P_{L}\circ P_{M}(\sigma(X))]~V^\top \\
&=& U \Diag P_{L}(\sigma(X))~V^\top \\
&=& P_{\sigma^{-1}(L)}(X),
\end{eqnarray*}
as claimed. The proof of the proposition in the case of symmetric matrices is similar. We leave the details to the reader.
\end{proof}

\section{Application: Density in Sobolev spaces}\label{app:sob}
In this section, we will always consider integration on $\R^n$ with respect to the Lebesgue measure. To make notation consistent with existing literature, in contrast to previous sections, we will let $p$ denote the order of Lebesgue spaces and we will let $k$ denote the degree of smoothness of a function. Let $w\colon\R^n\to\R$ be a locally integrable function satisfying $w(x) >0$ for almost all $x\in\R^n$; we will call $w$ a {\em weight}. Consider an open subset $\Omega$ of $\R^n$. Then for $1 \leq p <\infty$, define $L^p_w(\Omega)$ as the set of measurable functions $f$ on $\Omega$ satisfying 
$$\|f\|_{L^p_w(\Omega)}:=\Big(\int_{\Omega} |f|^p w \,dx\Big)^{\frac{1}{p}}<\infty.$$
Pairwise identifying functions in $L^p_w(\Omega)$ that are pairwise equal almost everywhere with respect to Lebesgue measure, the set $L^p_w(\Omega)$ becomes a Banach space with norm $\|\cdot\|_{L^p_w(\Omega)}$.

Consider a locally integrable function $f$ on $\Omega$. Then the {\em weak i'th partial derivative} of $f$ (for $i=1,\ldots,n$) is any other locally integrable function $v_i$ on $\Omega$ satisfying 
$$\int_{\Omega} f\frac{\partial \phi}{\partial x_i} \,dx=-\int_{\Omega}v_i\phi \,dx,\quad \textrm{ for any }\phi\in C_c^{\infty}(\Omega).$$
Weak derivatives are unique up to measure zero. We will denote the weak $i$'th  partial derivative of $f$ by $D^i f$. The vector of weak partial derivatives $(D^1 f,\ldots,D^n f)$ will be denoted by $D f$.

We assume that $w^{-1/(p-1)}$ is locally integrable on $\Omega$, which ensures that every $f\in L^p_w(\Omega)$ has weak partial derivatives, see \cite[Theorem~1.5]{Kufner/Opic:1984}. The {\em weighted Sobolev space} $W^{1,p}_w(\Omega)$ is defined by 
\[
W^{1,p}_w(\Omega)=\{f\in L^p_w(\Omega): D^i f\in L^p_w(\Omega) \textrm{ for } i=1,\ldots, n\},
\]
and it becomes a Banach space when equipped with the norm
\[
\|f \|_{W^{1,p}_w(\Omega)} := \Big(\|f\|_{L^p_w(\Omega)}^p+\sum_{i=1}^n \|D^{i}f\|_{L^p_w(\Omega)}^p\Big)^{\frac{1}{p}},
\]
see \cite[Theorem~1.11]{Kufner/Opic:1984}.

A basic question in the theory of Sobolev spaces is under which conditions the elements of $W^{1,p}_w(\Omega)$ can be approximated by more regular functions. In this direction Meyers and Serrin \cite{Meyers/Serrin:1964} (see also~\cite[Theorem~	3.17]{Adams/Fournier:2003}) famously showed that the set of $C^1$-smooth functions contained in $W^{1,p}_1(\Omega)$ is actually dense in $W^{1,p}_1(\Omega)$. More generally, if $w$ satisfies {\em Muckenhoupt's $A_p$ condition}, see \cite[Definition 1.2.2]{Turesson:2000fk}, it is true that $C^{\infty}$-smooth functions contained in $W^{1,p}_w(\Omega)$ are dense in $W^{1,p}_w(\Omega)$ \cite[Corollary 2.1.6]{Turesson:2000fk}. Ensuring that the approximating functions can be extended to all of $\R^n$ in a smooth way requires extra conditions on the geometry of $\Omega$. In particular, the following property holds in a variety of circumstances:
\begin{equation}\label{eq:cinf}
\textrm{The set of restrictions to } \Omega \textrm{ of functions in } C^{\infty}_c(\R^n) \textrm{ is dense in } W^{1,p}_w(\Omega).
\end{equation}
For example, in the unweighed case where $w=1$, equation (\ref{eq:cinf}) holds whenever $\Omega$ satisfies the so-called \emph{segment condition}, essentially stating that $\Omega$ is located on one side of its boundary. For more details, see~\cite[Definition~3.21 and Theorem~3.22]{Adams/Fournier:2003}. In the weighted case and assuming that $w$ satisfies the $A_p$~condition, it is sufficient that $\Omega$ be an \emph{$(\varepsilon,\delta)$ domain}, see~\cite{Chua:1992}.

Let now $Q$ denote the closure of $\Omega$, and assume it admits a Whitney (a)-regular $C^2$-stratification $\mathcal A$. Assume also that there is a subclass $\mathcal A_{\rm bd}\subset\mathcal A$ of strata whose union is equal to $\bd Q$. In particular, this is the case whenever $\mathcal A$ is a semi-algebraic stratification. Then, given that (\ref{eq:cinf}) holds, we can apply our previous results to identify a different class of functions that is also dense in $W^{1,p}(\Omega,w)$:
\[
C^k_{\rm Neu}(\Omega) :=
\left\{
\begin{array}{l}
\textrm{all restrictions to } \Omega \textrm{ of functions } g \in C^k_c(\R^n) \textrm{ with} \\[1mm]
\nabla g(x) \in T_xM \textrm{ for all } M\in\mathcal A_{\rm bd} \textrm{ and all } x\in M  
\end{array}
\right\}.
\]
The elements of $C^k_{\rm Neu}(\Omega)$ satisfy a Neumann boundary condition, namely that its gradients are tangent to the boundary of $\Omega$. In what follows, the $p$-norm on $\R^n$ will be denoted by $\|\cdot\|_p$.

\begin{theorem}
Fix a real number $1\le p< \infty$ and let $Q$ be the closure of an open set $\Omega\subset\R^n$. Assume that $Q$ admits a Whitney (a)-regular $C^2$-stratification~$\mathcal A$ such that $\bd Q$ consists of a collection $\mathcal A_{\rm bd}\subset \mathcal A$ of strata. If (\ref{eq:cinf}) holds, then $C^1_{\rm Neu}(\Omega)$ is dense in $W^{1,p}_w(\Omega)$.

If in addition $\mathcal A$ is normally flat and a $C^{k+1}$-stratification, then $C^k_{\rm Neu}(\Omega)$ is dense in $W^{1,p}_w(\Omega)$.
\end{theorem}
\begin{proof}
We start with the first assertion. Since (\ref{eq:cinf}) holds, it suffices to approximate functions in $C_c^{\infty}(\R^n)$. Let $f$ be such a function and let $V_1$ be a bounded open set containing its support. Additionally let $V_2$ be an open set, to be specified later, containing $\bd Q$. Consider now an arbitrary real number $\varepsilon>0$. An application of Theorem~\ref{thm:main} yields a $C^1$ function $g\in C^1_\text{Neu}(\Omega)$ such that its support is contained in $V_1$, it coincides with $f$ outside of $V_2$, the inequality $|f(x)-g(x)|<\varepsilon$ holds for all $x\in\R^n$, and we have the estimate $\lip g\leq c_1(\lip f)$ for a constant $c_1$ that only depends on the stratification~$\mathcal A$. This yields
\[
\int_\Omega |g(x)-f(x)|^p w(x)dx  \le \varepsilon \int_{V_2}w(x)dx
\]
and, since $\|\cdot\|_p \le c_2 \|\cdot\|$ for some constant $c_2$ that only depends on $p$ and $n$,
\[
\begin{array}{ll}
\int_\Omega \|D g(x) - D f(x)\|_p^p w(x)dx &\le 2^{p-1} \int_{V_2\cap\Omega} \left(\|D g(x)\|^p_p + \|D f(x)\|_p^p\right) w(x)dx \\
&\le  2^{p-1} c_2^p\int_{V_2\cap\Omega} \left(\|D g(x)\|^p + \|D f(x)\|^p\right) w(x)dx \\
&\le 2^{p-1} c_2^p (c_1^p + 1) (\lip f)^p\int_{V_2}w(x)dx.
\end{array}
\]
Since $w$ is locally integrable, $w(x)dx$ is a Radon measure on $\R^n$, hence outer regular. Since also $\bd Q$ is the union of finitely many $C^2$~manifolds of dimension at most $n-1$, and hence a nullset, $V_2$ can be chosen so that we have $\bd Q\subset V_2$ and
\[
\int_{V_2}w(x)dx\le \min\{1,\  \varepsilon 2^{1-p}c_2^{-p}(c_1^p+1)^{-1}(\lip f)^{-p}\}.
\]
With this choice of $V_2$, the inequality $\|g-f\|^p_{W^{1,p}_w(\Omega)} \le 2\varepsilon$ holds. This finishes the proof of the first assertion. To prove the second assertion, we simply apply Theorem~\ref{T:smooth} instead of Theorem~\ref{thm:main} in the above proof. This ensures that $g$ is $C^k$-smooth.
\end{proof}

\begin{remark}
{\rm
Unfortunately our results are not strong enough to obtain analogous results for higher order Sobolev spaces $W^{k,p}_w(\Omega)$, where $k>1$. The reason is that the size of higher-order derivatives of $g$ cannot be controlled in terms of those of $f$ using our current techniques.}
\end{remark}

\section{Application: matrices with nonnegative determinant} \label{S:M+}
In this section, we outline how one may use Theorem~\ref{T:smooth} in a concrete situation. The full details may be found in \cite{Larsson:2012}. Before proceeding though, we give a rather simple and very intuitive motivation for why our main results may have significant applicability. Consider a $C^2$-smooth function $f\colon\R^n\to\R$ and a $C^1$-smooth function $h\colon\R^n\to\R$. As is well known, for many subsets $Q$ of $\R^n$, the integration-by-parts formula
\begin{equation}\label{eqn:int_by_parts}
\int_Q{h\Delta f} \,dx=\int_{\sbd Q} h\langle\nabla f,\hat{n}\rangle \,dS -\int_Q\langle\nabla h,\nabla f \rangle \, dx,
\end{equation}
holds, where $\hat{n}$ is a well-defined outward normal vector and $dS$ denotes an appropriate surface area measure. The boundary term is often the troublesome term in the expression above. Theorem~\ref{T:smooth}, on the other hand, implies that under suitable conditions on $Q$ we may approximate $f$ by another $C^2$-smooth function $g$, for which the corresponding boundary term vanishes. 

\begin{corollary}[Simplified integration-by-parts]\label{cor:int}{\hfill}{\\}
Consider a subset $Q\subset\R^n$, a $C^2$-smooth function $f\colon\R^n\to\R$, and a $C^1$-smooth function $h\colon\R^n\to\R$. Suppose in addition that the equation (\ref{eqn:int_by_parts}) holds and that~$Q$ admits a normally flat $C^{p+2}$-stratification. Then for any $\epsilon>0$, there exists a $C^{p+1}$-smooth function $g\colon\R^n\to\R$ with $|f-g|<\epsilon$, that satisfies the simplified integration-by-parts formula
\[
\int_Q{h\Delta g} \,dx= -\int_Q\langle\nabla h,\nabla g \rangle \, dx.
\]
If in addition (\ref{eq:cinf}) holds for some weight $w$, then we may ensure $\|f -g\|_{W^{1,q}_w(\Omega)}<\epsilon$ for any fixed $q\in [1,\infty)$.
\end{corollary}

The corollary above is the driving force behind the specific application we consider in this section, though we defer most of the details to \cite{Larsson:2012}. We now begin the development. To this end, consider the set
\[
\mathbf M^{n\times n}_+ := \left\{ X\in \mathbf M^{n\times n}: \det X \ge 0\right\},
\]
where $\mathbf M^{n\times n}$ is endowed with the trace product and the Frobenius norm. The reasons for studying this set are two-fold: firstly, it was one of the original motivations for developing the results in this paper. This is due to its role in a certain application involving stochastic processes, which we outline below. Secondly, and more importantly in the context of the present paper, it allows us to illustrate how the components of the iterative construction in Section~\ref{S:it} can be used more generally. Specifically, our goal is to establish an improved version of Theorem~\ref{T:smooth}, discussed below.

We first need a few preliminaries. We begin by observing that Propositions~\ref{P:polflat} and Theorem~\ref{P:A1sp} show that the collection of $C^{\infty}$ manifolds
\[
\quad M_k := \{X\in\mathbf M^{n\times n} : \rk X = k\} \quad\textrm{ for } k=0,\ldots,n,
\]
is a normally flat $C^{\infty}$-stratification of $\mathbf M^{n\times n}$. 
Consequently, the collection 
\[
\mathcal A_+ := \{M_0,\ldots,M_{n-1},M_{n+}\},
\]
where $M_{n+}$ is the open set \[M_{n+}:=\{X\in\mathbf M^{n\times n} : \det X>0\},\]
is a normally flat ${C^{\infty}}$-stratification of $\mathbf M^{n\times n}_+$.

For a point $x\in\R^n$, define a new point $x^+\in\R^n$ by setting $x^+_i=\frac{1}{x_i}$ if $x_i\ne 0$, and $x^+_i=0$ otherwise. 
Then if a matrix $X\in \mathbf M^{n\times n}$ has a decomposition $X=U\Diag x\ V^\top$ with $U,V\in O(n)$, the \emph{Moore-Penrose inverse} of $X$ is given by $X^+ := V \Diag x^+\ U^\top$. The transpose of the Moore-Penrose inverse, which we denote by
\[
X^{\mp} := (X^+)^\top = (X^\top)^+,
\]
will play a role in our development. The main result of this section is the following.

\begin{theorem}[Approximation on matrices with positive determinant] \label{T:pr11}
Consider the set of matrices $Q=\mathbf M^{n\times n}_+$, along with the $C^{\infty}$-stratification $\mathcal A_+$, and let $f$ and $\varepsilon$ be as in Theorem~\ref{T:smooth}. Then there is a function $g:\R^n\to \R$ that satisfies all the properties in Theorem~\ref{T:smooth}, as well as the property that the mapping
\[
X \mapsto \left\langle X^\mp, \nabla g(X) \right\rangle,
\]
is continuous.
\end{theorem}

\begin{remark}{\rm
In the case $\det X>0$, that is when $X$ is invertible, we have $X^\mp=(X^{-1})^\top=\nabla \ln\det (X)$.}
\end{remark}

The importance of Theorem~\ref{T:pr11} comes from its significance in the study of a class of stochastic processes called the \emph{matrix-valued Bessel processes}, introduced in~\cite{Larsson:2012}. These are Markov processes whose infinitesimal generator arises as an extension of the differential operator
\[
\mathcal L = \frac{1}{2}\Delta + \frac{\delta-1}{2}\langle \nabla \ln\det(X), \nabla\rangle,
\]
acting on a suitable class of functions on $\mathbf M^{n\times n}_+$. Here $\delta>0$ is fixed. A consequence of Theorem~\ref{T:pr11} is that any compactly supported Lipschitz continuous function on $\mathbf M^{n\times n}_+$ can be approximated uniformly by some compactly supported $C^2$ function $g$ that satisfies a Neumann boundary condition, and has the crucial property that $\mathcal Lg$ is bounded (in fact, the boundedness of $\mathcal Lg$ is the incremental benefit that motivates Theorem~\ref{T:pr11}.)  For such functions $g$, it is possible to obtain an integration-by-parts formula,
\[
\int_Q h(X) \mathcal Lg(X) w(X)dX = - \frac{1}{2}\int_Q \langle \nabla h(X),\nabla g(X) \rangle w(X)dX,
\]
where $w(X)=(\det X)^{\delta-1}$ can be shown to be locally integrable on ${\bf M}^{n\times n}_+$, and~$h$ varies over some class of test functions. This is perhaps not so surprising in light of Corollary~\ref{cor:int}. The right-hand-side of the above equation extends to a \emph{Dirichlet form} $\mathcal E(g,h)$ on the weighted space $L^2_w(\mathbf M^{n\times n}_+)$. This allows one to apply standard existence results for the associated stochastic process, which we denote by ${\bf X}=({\bf X}_t:t\ge 0)$. Furthermore, for a function $g$ that satisfies the integration-by-parts formula, and for which $\mathcal Lg$ is bounded, the process
\[
g({\bf X}_t) - g({\bf X}_0) - \int_0^t \mathcal L g({\bf X}_s)ds, \qquad t\ge 0,
\]
is a martingale.  Since, by Theorem~\ref{T:pr11}, this holds for a large class of functions $g$, it is possible to obtain a good description of the probability law of $\bf X$. In particular, powerful uniqueness results become available. For a detailed discussion, see \cite{Larsson:2012}.

After this brief digression, let us proceed to the proof of Theorem~\ref{T:pr11}. The following lemma establishes a key property of the Moore-Penrose inverse.

\begin{lemma}[Projection of the Moore-Penrose inverse]\label{L:MPinv}
Consider a stratum $M\in\mathcal A_+$ and a matrix $X\in \mathbf M^{n\times n}$ that when projected onto $M$ yields a single matrix $Y$. Then the equality 
\[
P_{T_YM}(X^\mp) = Y^\mp \quad\textrm{holds}.
\]
\end{lemma}

\begin{proof}
Let $k$ be the rank of the matrices in $M$, and note that since $P_M(X)$ is a singleton, the inequality $\rk X\ge k$ must hold.
Let $X=U\Sigma V^\top$ be a singular value decomposition of $X$, and let $\Sigma_k$ be the diagonal matrix obtained from $\Sigma$ by setting all but the $k$ largest entries to zero. Then we have the equality $Y=U\Sigma_k V^\top$, as is apparent from Lemma~\ref{L:PA1sp}. Observe,
\[
X^\mp = Y^\mp  +  \left(X^\mp - Y^\mp\right) = U \Sigma_k^+ V^\top + U\left( \Sigma^+ - \Sigma_k^+\right) V^\top.
\]
The desired result follows once we show that the right-hand-side is the decomposition of $X^\mp$ as a direct sum in $T_YM\oplus N_YM$. To do this, first note that for small~$t$, the matrix
\[
Y +  tU \Sigma_k^+ V^\top = U (\Sigma_k +  t \Sigma_k^+) V^\top,
\]
has rank $k$, and hence lies in $M$. Consequently the inclusion $U \Sigma_k^+ V^\top\in T_YM$ holds. Next, for small $t$ we have 
\[
P_M\left( Y + tU\left( \Sigma^+ - \Sigma_k^+\right) V^\top \right) = P_M\left( U\left( \Sigma_k + t(\Sigma^+ -\Sigma_k^+) \right) V^\top\right) = Y,
\]
thus verifying that the matrix $U\left( \Sigma^+ - \Sigma_k^+\right) V^\top$ lies in $N_YM$, as required.
\end{proof}

Next, we establish the validity of the following inductive step. We will use the well-known property that the Moore-Penrose inverse mapping $X\mapsto X^{+}$ is continuous on each stratum $M$ of~$\mathcal{A}_{+}$.

\begin{proposition} \label{P:Mplus}
Consider the setup and notation given in Subsection~\ref{S:it1}, with $\mathcal A=\mathcal A_+$. Assume that the following properties hold.
\begin{enumerate}
\item $U_\delta$ is contained in the set $\widehat U$ from Lemma~\ref{L:A1},
\item $f$ is $C^1$ smooth, and we have $\nabla f(X)\in T_X\Gamma$ for every $X\in\Gamma$,
\item each stratum $L\in\mathcal A_+$ with $L\subset\Gamma$ has a neighborhood $V$ such that
\[
f(X) = f\circ P_L(X) \quad \textrm{for all} \quad X\in V,
\]
\item the map $X \mapsto \langle \nabla f(X), X^\mp\rangle$ is continuous at every $\overline X\in \Gamma$.
\end{enumerate}
Define a function $g=\phi f\circ P_M + (1-\phi)f$.  Then the mapping
\[
X \mapsto \langle \nabla g(X), X^\mp\rangle,
\]
is continuous at every $\overline X\in \Gamma\cup M$.
\end{proposition}

\begin{proof}
First note that, by Proposition~\ref{P:grad1}, the function $g$ is again $C^1$ smooth, with $\nabla f=\nabla g$ on $\Gamma$.
Consider a matrix $\overline X\in M$, and write $Y=P_M(X)$ as usual for $X$ near $M$. Since we have $g=f\circ P_M$ near $M$, we deduce
\[
\langle \nabla g(X), X^\mp \rangle = \langle \nabla f(Y), \mathrm DP_M(X)X^\mp\rangle.
\]
Using Lemma~\ref{L:MPinv} and the observation $\mathrm DP_M(X)U=0$ for $U\in N_YM$, we obtain
\[
\mathrm DP_M(X)X^\mp = \mathrm DP_M(X)P_{T_YM}(X^\mp) = \mathrm DP_M(X)Y^\mp.
\]
From the convergence $X\to\overline X$, we deduce $Y\to \overline X$, and consequently $Y^\mp\to \overline X^\mp$, since the matrices $Y$ and $\overline X$ have the same rank. We successively conclude,
\begin{eqnarray*}
\lim_{X\to\overline X} \langle \nabla g(X), X^\mp \rangle
&=& \langle \nabla f(\overline X), \mathrm DP_M(\overline X)\overline X^\mp\rangle \\
&=& \langle \nabla(f\circ P_M)(\overline X), \overline X^\mp\rangle \\
&=& \langle\nabla g(\overline X),\overline X^\mp\rangle,
\end{eqnarray*}
thereby verifying continuity at $\overline X$. Consider now a stratum $L\subset \Gamma$ and a point $\overline X\in L$. Since we have $\nabla f=\nabla g$ outside $U_\delta$, we can assume $X\in U_\delta$. We obtain
\begin{align}
 \label{eqn:stuff}
\langle \nabla g(X), X^\mp\rangle - \langle \nabla f(X), X^\mp\rangle
&= \phi(X) \langle \nabla(f\circ P_M)(X) - \nabla f(X), X^\mp\rangle \\
&\qquad + \left(f\circ P_M(X) - f(X) \right) \langle \nabla \phi(X), X^\mp\rangle.\nonumber
\end{align}
Observe that whenever $X$ is close enough to $\overline X$, we have
\[
f\circ P_M(X)=f\circ P_L\circ P_M(X)=f\circ P_L(X)=f(X),
\]
due to the assumption on $f$ and normal flatness of $\mathcal A_+$. Consequently, the right-hand-side of (\ref{eqn:stuff}) vanishes as soon as $X$ gets sufficiently close to $\overline X$. The result follows.
\end{proof}

\begin{proof}[Proof of Theorem~\ref{T:pr11}]
We simply carry out the same iterative procedure as in the proofs of Theorems~\ref{thm:main} and~\ref{T:smooth} (see Subsection~\ref{S:pf}), while in addition applying Proposition~\ref{P:Mplus} in each step. Note that each time this is done, the hypotheses $(2)$ and $(3)$ will be satisfied since Propositions~\ref{P:grad1} and~\ref{P:Cp} were applied in some earlier iteration (initially these hypotheses are vacuously true.) The hypothesis $(1)$ is easily achieved by shrinking $U_\delta$ if necessary, and $(4)$ holds due to earlier applications of Proposition~\ref{P:Mplus}.
\end{proof}

\begin{acknowledgements*}\label{ackref}
Thanks to Adrian S. Lewis and James Renegar for encouraging the authors to undertake the current work, and to Hedy Attouch, J\'{e}r\^{o}me Bolte, and Guillaume Carlier for numerous suggestions of possible further extensions of the work presented here. The authors would also like to thank W. Zachary Rayfield for stimulating discussions. Finally, many thanks are due to an anonymous referee, whose detailed reading of the manuscript led to substantial improvements.
\end{acknowledgements*}


\bibliographystyle{amsplain}
\parsep 0pt
\bibliography{bibl}

\end{document}